\documentclass[reqno]{amsart}

\usepackage{amsmath,amssymb,amsfonts,bbm,stmaryrd}

\usepackage[arrow,matrix,curve]{xy}  
\usepackage{latexsym}
\usepackage{graphics}

\usepackage{epsfig}
\usepackage{color}
\usepackage{amsthm}
\usepackage{amssymb}
\usepackage{amsmath}
\usepackage{amscd}
\usepackage{amsopn}

\usepackage{url}
\usepackage{hyperref}\hypersetup{colorlinks}

\usepackage{color} %\textcolor{red}{Test}

\definecolor{darkred}{rgb}{0.8,0,0} %can change the intensity in [0,1]
\definecolor{darkgreen}{rgb}{0,0.6,0}
\definecolor{darkblue}{rgb}{0,0,0.8}

\hypersetup{colorlinks,linkcolor=darkblue,filecolor=darkgreen,urlcolor=darkred,citecolor=darkgreen}

\newtheorem{thm}{Theorem}[section]
\newtheorem{lemma}[thm]{Lemma}
\newtheorem{prop}[thm]{Proposition}
\newtheorem{cor}[thm]{Corollary}
\theoremstyle{definition}
\newtheorem{definition}[thm]{Definition}
\theoremstyle{remark}
\newtheorem{rmk}[thm]{Remark}

\newcommand{\eps}{\epsilon}

\newcommand{\R}{\mathbb{R}}

\def    \CF     {\operatorname{CF}}\def    \HF     {\operatorname{HF}}
\def    \CF     {\operatorname{CF}}
\newcommand{\CZ}{\text{CZ}}
\newcommand{\SDM}{symplectically degenerate maximum}
\newcommand{\SDMs}{symplectically degenerate maxima}

\numberwithin{equation}{section}

\begin{document}
%Titel und Formangaben (Datum und Co)

\title[Conley conjecture for irrational symplectic manifolds]{The Conley conjecture for irrational symplectic manifolds}

\author[Doris Hein]{Doris Hein}

\address{Department of Mathematics, UC Santa Cruz,
Santa Cruz, CA 95064, USA}
\email{dhein@ucsc.edu}

\date{\today}
\maketitle

\textbf{Abstract:}
We prove a generalization of the Conley conjecture: Every Hamiltonian diffeomorphism of a closed symplectic manifold has infinitely many periodic orbits if the first Chern class vanishes over the second fundamental group. In particular, this removes the rationality condition from similar theorems by Ginzburg and G\"urel.
The proof in the irrational case involves several new ingredients including the definition and the properties of the filtered Floer homology for Hamiltonians on irrational manifolds. 
We also develop a method of localizing the filtered Floer homology for short action intervals using a direct sum decomposition, where one of the summands only depends on the behavior of the Hamiltonian in a fixed open set.

\tableofcontents
\newpage
%Einleitung
\section{Introduction}
\label{sec:intro}
\subsection{Main results}
\label{sec:Conley}

In this paper we prove that a Hamiltonian difffeomorphism of a symplectic manifold has infinitely many periodic orbits if the first Chern class of the manifold vanishes over $\pi_2(M)$.
This is a generalization of a conjecture Conley stated in 1984 in \cite{Co} for $M=T^{2n}$.
This conjecture was proved for the weakly non-degenerate Hamiltonian diffeomorphisms of tori in \cite{CZ} and of symplectically aspherical manifolds in \cite{SZ}. In \cite{FrHa}, the conjecture was proved for all Hamiltonian diffeomorphisms of surfaces other than $S^2$. In its original form, as stated 
in \cite{Co} for $M = T^{2n}$, the conjecture was established in \cite{Hi} and the case 
of an arbitrary closed, symplectically aspherical manifold was settled in \cite{Gi}. This proof was extended to symplectically rational manifolds $M$ with $c_1(M)|_{\pi_2(M)}=0$ in~\cite{GG-gaps}.

The main result of this paper is the following theorem:

\begin{thm} [Conley Conjecture, the irrational case]
Let $(M^{2n},\omega)$ be a closed symplectic manifold with 
$c_1(M)|_{\pi_2(M)}=0$ and let $\varphi$ be a Hamiltonian diffeomorphism on $M$ such that the fixed points of $\varphi$ are isolated.
Then $\varphi$ has simple periodic orbits of arbitrarily large period.
\label{ConlConj}
\end{thm}

In particular, this theorem implies all the results mentioned above. It also implies that on such manifolds any Hamiltonian diffeomorphism has infinitely many periodic orbits as there are infinitely many one-periodic orbits or we have simple periodic orbits of arbitrarily large period.

\begin{rmk}
The requirements on the manifold for theorems of this type have been relaxed more and more in the last years. However, the example of an irrational rotation on $S^2$ shows that the restrictions on the symplectic manifold cannot be completely eliminated.
\end{rmk}

As in \cite{GG-gaps}, it suffices to prove the result in the presence of a \SDM.

\begin{definition}
\label{SDM}
An isolated capped $k$-periodic orbit $\bar{x}$ of a $k$-periodic Hamiltonian $H$ is called a
 \textit{\SDM} of $H$ if \[\Delta_H(\bar{x})=0 \text{ and } \HF_n(H,\bar{x})\neq 0.\]
\end{definition}

In this definition, we denote by $\HF_n(H,\bar{x})$ the local Floer homology of $H$ at $\bar{x}$ and let $\Delta_H(\bar{x})$ be the mean index of $\bar{x}$.
We refer to~\cite{GG-loc} for details on the local Floer homology and to \cite{SZ} for the definition of the mean index.

This definition was first used in \cite{Hi} when the concept of \SDMs\ was introduced. It was explicitly stated and further investigated in \cite{Gi,GG-loc}. See also Proposition~\ref{SDM geo} for the geometric description of a \SDM.

\begin{thm}[Degenerate Conley Conjecture, the irrational case]
Let $(M^{2n},\omega)$ be a closed symplectic manifold with 
$c_1(M)|_{\pi_2(M)}=0$ and let $H$ be a Hamiltonian on $M$ such that the fixed points of $\varphi_H$ are isolated.
Assume, in addition, that $H$ has a \SDM.
Then the Hamiltonian diffeomorphism $\varphi_H$ generated by $H$ has geometrically distinct periodic orbits of arbitrarily large period.
\label{degConlConj}
\end{thm}

\begin{proof}[Proof of Theorem \ref{ConlConj}]
If there is no symplectically degenerate maximum, all one-periodic orbits have non-zero mean index. As the mean index is independent of cappings and grows linearly with iteration, the support of local Floer homology is shifting away from the interval $\left[0,2n\right]$, i.e. the local Floer homology in those degrees eventually becomes 0. The theorem follows then by a standard argument as in~\cite{Gi,GG-gaps,Hi,SZ}.
If there is a symplectically degenerate maximum, the theorem follows from Theorem \ref{degConlConj}.
\end{proof}

The proof of Theorem \ref{degConlConj} is based on a Floer theoretical argument establishing

\begin{thm}
\label{theorem}
Let $(M,\omega)$ be weakly monotone and closed.
Assume that $H$ has a \SDM\ at $\bar{x}$ with
$\mathcal{A}_H(\bar{x})=c$. 
Then  for every sufficiently small $\epsilon >0$ there exists some $k_{\epsilon}$ such that
\[
\HF_{n+1}^{(kc+\delta_k,\,kc+\epsilon)}(H^{(k)})\neq 0 \text{  for every } k>k_{\epsilon} \text{ and some } \delta_k\in(0,\epsilon).
\]
\end{thm}

Here $H^{(k)}$ denotes the one-periodic Hamiltonian $H$ viewed as $k$-periodic function for some integer $k$, see also Section~\ref{sec:sympl-Ham}.
This theorem implies Theorem~\ref{degConlConj} and thus also Theorem~\ref{ConlConj}.

\begin{proof}[Proof of Theorem \ref{degConlConj}]
Arguing by contradiction, we assume that for every sufficiently large period all periodic orbits are iterated and let $k$ be a sufficiently large prime. Then every $k$-periodic orbit is an iterated one-periodic orbit. By Theorem~\ref{theorem}, there exists a capped $k$-periodic orbit $\bar{y}_k$ with

\[
1 \leq \Delta_{H^{(k)}}(\bar{y}_k) \leq 2n+1.
\]
We only have finitely many one-periodic orbits and $c_1(M)|_{\pi_2(M)}=0$. The finitely many non-zero mean indices grow linearly with the order of iteration and thus for sufficiently large $k$ the indices of the iterations shift away from the interval $\left[1,\,2n+1\right]$. The iterations of the orbits with mean index zero have also mean index zero.
Thus the orbit $\bar{y}_k$ cannot be an iterated one-periodic orbit in contradiction to the choice of $k$. 
\end{proof}

The proof of Theorem~\ref{theorem} builds on a generalization of the methods from the proof of the rational case in \cite{GG-gaps}. 
We define filtered Floer homology for symplectically irrational manifolds and prove that this homology can be localized for sufficiently small action intervals.
The localization is realized by a direct sum decomposition for the filtered Floer homology groups similar to the one existing in the rational case. 
To this end, we use ideas from \cite{Us} to bound the energy of Floer trajectories.

\subsection{Action and index gap}

Theorem~\ref{theorem} can also be sed to control the behavior of actions and mean indices of periodic orbits, cf. \cite{GG-gaps}.
To state the results, we need to introduce some notation.
We call the difference $\mathcal A_{H^{(k)}}(\bar x)-\mathcal A_{H^{(k)}}(\bar y)$ the \textit{action gap} between the two capped $k$-periodic orbits $\bar x$ and $\bar y$. Similarly, the \textit{mean index gap} between the two orbits is the difference $\Delta_{H^{(k)}}(\bar x)-\Delta_{H^{(k)}}(\bar y)$. Both can be zero, even for geometrically distinct orbits $x$ and $y$.
The \textit{action-index gap} between $\bar x$ and $\bar y$ is the vector in $\mathbb{R}^2$ whose components are the action gap and the mean index gap.

Recall also that an increasing sequence of integers $\nu_1<\nu_2<\ldots$ is called \textit{quasi-arithmetic} if the differences $\nu_{i+1}-\nu_i$ are bounded by a constant, which is independent of $i$.

\begin{thm}[Bounded gap theorem]
\label{thm:gap1}
  Let $H$ be a Hamiltonian on a closed
  symplectic manifold $(M^{2n},\omega)$ with $N\geq 2n$ such that all periodic orbits
  of $\varphi_H$ are isolated.

Then there exists a capped one-periodic orbit $\bar x$ of $H$, a
quasi-arithmetic sequence of iterations $\nu_i$, and a sequence of capped
$\nu_i$-periodic orbits $\bar y_i$, geometrically distinct from
$\bar x^{\nu_i}$, such that the sequence of action--index gaps
\begin{equation*}
\label{eq:gap}
\left(\mathcal A_{H^{(\nu_i)}}(\bar x^{\nu_i})-\mathcal A_{H^{(\nu_i)}}(\bar y_i),
\Delta_{H^{(\nu_i)}}(\bar x^{\nu_i})-\Delta_{H^{(\nu_i)}}(\bar y_i)\right)
\end{equation*}
is bounded.
\end{thm}

This is a generalization of a result in \cite{GG-gaps} where the theorem was proved in the case of a symplectically rational manifold. 
The proof of the present version follows the same path as the argument in \cite{GG-gaps}, utilizing the general form of Theorem~\ref{theorem}.
The new point is that we now use the mean index bound, rather than the action bound as in \cite{GG-gaps}, to show that the orbits $\bar y_i$ and the iterations of $\bar x$ are geometrically distinct.
As in the rational case, this theorem implies the following corollary in the generalized situation.
 
\begin{cor}
\label{cor:gap1}
Let $M$ and $H$ be as in Theorem \ref{thm:gap1}. Then there exists a
quasi-arithmetic sequence of iterations $\nu_i$ and sequences of
geometrically distinct $\nu_i$-periodic orbits $\bar z_i$ and $\bar z'_i$
such that the sequence of action--index gaps between $\bar z_i$ and
$\bar z'_i$ is bounded.
\end{cor}

\subsection{Organization of the paper}
\label{sec:orga}

In Sections \ref{sec:sympl-Ham} and \ref{sec:perorb-HF} we set the notation and conventions used in this paper. Section \ref{sec:filtHF} is devoted to the discussion of the construction of filtered Floer homology. In particular, we focus on the construction of filtered Floer homology that is necessary for degenerate Hamiltonians on irrational symplectic manifolds.
In Section \ref{sec:direct sum} we establish a localization of Floer homology via a direct sum decomposition in filtered Floer homology.
Finally, we prove Theorem~\ref{theorem} in Section \ref{sec:Beweis} using the direct sum decomposition from Proposition \ref{hom decomp}.

\subsection{Acknowledgments}
The author is deeply grateful to Viktor Ginzburg for showing her this problem, helping with the solution and his kind, thoughtful advice.
She would also like to thank the referee for useful remarks and Katrin Wehrheim for useful discussions.

\section{Preliminaries}
\label{sec:prelims}

In this section we will set the notation used in this paper and review some of the basic facts needed in order to prove the theorems.

\subsection{Symplectic manifolds and Hamitonian flows}
\label{sec:sympl-Ham}

Let $(M^{2n},\,\omega)$ be a closed symplectic manifold of dimension $2n$ with first Chern class $c_1(M)$ and the minimal Chern number $N$. 
Throughout the paper we assume $(M,\,\omega)$ to be \textit{weakly monotone}, i.e. $N\geq n-2$ or $\left[\omega\right]|_{\pi_2(M)}=\lambda c_1(M)|_{\pi_2(M)}$ for some non-negative constant $\lambda$. In particular we will focuss on the first case as in the latter case the manifold is rational and the theorems are already proved in \cite{GG-gaps}.

All considered Hamiltonians $H$ on $M$ are assumed to be one-periodic in time, i.e. functions $H\colon S^1\times M\to \mathbb{R}$ where $S^1=\mathbb{R}/\mathbb{Z}$ and we will set $H_t(x)=H(t,\,x)$.
A one-periodic Hamiltonian $H$ can also be viewed as $k$-periodic for any integer $k$. For our argument it is sometimes crucial to keep track of the period we are interested in. If a one-periodic Hamiltonian $H$ is viewed as $k$-periodic, we refer to it as the \textit{$k$th iteration} of $H$ and denote it by $H^{(k)}$.

As the symplectic form $\omega$ is non-degenerate, the Hamiltonian equation $i_{X_H}\omega=-dH$ gives rise to a well-defined Hamiltonian vector field $X_H$. The time-1-map of the flow of the Hamiltonian vector field $X_H$ is called a Hamiltonian diffeomorphism and denoted by $\varphi_H$.

The composition $\varphi_H^t\circ\varphi_K^t$ of two Hamiltonian flows is again Hamiltonian. It is generated by
\[
(K\# H)_t=K_t+H_t\circ \varphi_K^{-t}.
\]
In general, this function need not be one-periodic, even if both $H$ and $K$ are one-periodic Hamiltonians. But $K\# H$ will be one-periodic if both are one-periodic and in addition $K$ generates a loop of Hamiltonian diffeomorphisms. This will always be the case in this paper.

\subsection{Capped periodic orbits and Floer homology}
\label{sec:perorb-HF}

Let $x\colon S^1\to M$ be a contractible loop. A \textit{capping} of $x$ is defined to be a map $u\colon D^2\to M$ such that $u|_{S^1}=x$. Two cappings are called \textit{equivalent} if the integrals over the symplectic form $\omega$ and the first Chern class $c_1(M)$ over the two capping discs agree. We denote the pair $(x,\,\left[u\right])$ of a loop $x$ with an equivalence class of cappings $\left[u\right]$ by $\bar{x}$. In the symplectically aspherical case, all cappings are equivalent.

\subsubsection{Hamiltonian action and the mean index}
The Hamiltonian action functional is defined by
\[
\mathcal{A}_H(\bar{x})=-\int_u\omega +\int_{S^1}H_t(x(t))\ dt
\]
on the space of capped closed loops.
This space is a covering space of the space of contractible loops and the critical points of the action functional are exactly the capped one-periodic orbits of the Hamiltonian vector field $X_H$. The set of critical values of the action is denoted by $\mathcal{S}(H)$ and called the \textit{action spectrum} of $H$.

In this paper we only work with contractible periodic orbits and every periodic orbit is assumed to be contractible, even if this is not explicitly stated.

A one-periodic orbit $x$ of $H$ is said to be \textit{non-degenerate} if the linearized return map $d\varphi_H\colon T_{x(0)}M\to T_{x(0)}M$ does not have one as an eigenvalue. Following \cite{SZ}, we call an orbit \textit{weakly non-degenerate} if at least one eigenvalue is not equal to one and \textit{strongly degenerate} otherwise. We refer to a Hamiltonian $H$ as \textit{non-degenerate}, if all its one-periodic orbits are non-degenerate.

In general, the mean index and the action depend on the equivalence class of the capping $u$ of the loop $x$. More concretely let $A$ be an embedded 2-sphere and denote by $\bar{x}\# A$ the recapping of $\bar{x}$ by attaching $A$. Then we have
\[
\mathcal{A}_H(\bar{x}\# A)=\mathcal{A}_H(\bar{x})-\int_A\omega
\]
by the above formula for the Hamiltonian action.
For the definition and properties of the mean index we refer the reader to \cite{SZ}. A list of properties of the mean index can also be found in \cite{GG-gaps}. Here we only mention that the mean index $\Delta(\bar x)$ depends on the capping via
\[
\Delta(\bar{x}\# A)=\Delta(\bar{x})-2c_1(A).
\]

The $k$th iteration of a capped orbit $\bar{x}$ carries a natural capping and with that capping it is denoted by $\bar{x}^k$. The mean index and the action both are homogeneous with respect to iteration and satisfy the iteration formulas
\[
\mathcal{A}_{H^{(k)}}(\bar{x}^k)=k\mathcal{A}_H(\bar{x}) \text{ and } \Delta_{H^{(k)}}(\bar{x}^k)=k\Delta_H(\bar{x}).
\]

\subsubsection{Floer homology}
Up to sign we define the Conley-Zehnder index as in \cite{Sa,SZ} and use the normalization such that for a non-degenerate maximum $\gamma$ of an autonomous Hamiltonian with small Hessian we have $\mu_{\CZ}(\gamma)=n$; see \cite{GG-gaps}.

We define the Floer homology for a non-degenerate Hamiltonian $H$ as in \cite{Sa,SZ}. The homology is graded by the Conley-Zehnder index.
The Floer chain groups are generated by the capped one-periodic orbits of $H$ and the boundary operator is defined by counting solutions to the Floer equation
\begin{equation}
\frac{\partial u}{\partial s}+J_t(u)\frac{\partial u}{\partial t}=-\nabla H_t(u)
\label{Floer}
\end{equation}
with finite energy. 
As is well known, Floer trajectories for a non-degenerate Hamiltonian $H$ with finite energy converge to periodic orbits $\bar{x}$ and $\bar{y}$ as $s$ goes to $\pm\infty$ and satisfy
\[
E(u)=\mathcal{A}_H(\bar{x})-\mathcal{A}_H(\bar{y})=\int_{-\infty}^{\infty}\int_{S^1}\left\|\frac{\partial u}{\partial s}\right\|^2\ dt\ ds.
\] 
The boundary operator counts Floer trajectories converging to periodic orbits $y$ and $x$ as $s\to\pm\infty$ and satisfying the condition $\left[(\text{capping of } \bar{x})\# u\right]=\left[\text{capping of } \bar{y}\right]$.
This construction extends by continuity from non-degenerate Hamiltonians to all Hamiltonians, see \cite{Sa}.

For two non-degenerate Hamiltonians $H^0$ and $H^1$, a homotopy from $H^0$ to $H^1$ induces a homomorphism of chain complexes which gives an isomorphism between the Floer homologies $\HF_{*}(H^0)$ and $\HF_{*}(H^1)$ which is independent of the choice of homotopy. This map is defined analogously to the Floer boundary operator using a version of the Floer equation with the homotopy $H^s$ on the right hand side.

The \textit{local Floer homology} $\HF_{*}(H,\bar{x})$ of a Hamiltonian $H$ at a capped one-periodic orbit $\bar{x}$ is also defined as usual, see \cite{Gi, GG-gaps,GG-loc}. As the action does not enter the argument, the definition goes through in the irrational case. It is constructed using a small non-degenerate perturbation of the Hamiltonian in a neighborhood of  $x$. For a more detailed definition and a discussion of the properties of local Floer homology see \cite{Gi,GG-gaps,GG-loc}.

\subsection{Filtered Floer Homology}
\label{sec:filtHF}

In this section, we give a definition of filtered Floer homology for degenerate Hamiltonians on symplectically irrational manifolds. 

As the action decreases along Floer trajectories of a non-degenerate Hamiltonian $H$, we also have well-defined chain complexes that only involve orbits with action in an interval $(a,\,b)$ if $a$ and $b$ are not in the action spectrum $\mathcal S(H)$. This complex gives rise to the \textit{filtered Floer homology} $\HF_{*}^{(a,\,b)}(H)$. This construction extends by continuity to degenerate Hamiltonians if the manifold is rational, since in this case the Floer homology is independent of the choice of a sufficiently small, non-degenerate perturbation.

In the case of an irrational manifold, the action filtration of Floer homology for degenerate Hamitonians cannot be unambigously defined simply by  continuity as the resulting groups depend very sensitively on the non-degenerate perturbation. We thus use the following construction for the filtered Floer homology, which in the case of a rational manifold gives the same homology groups as continuity:

Let $H$ be a fixed Hamiltonian on $M$.
To define $\HF_{*}^{(a,\,b)}(H)$, consider perturbations $K$ of $H$ with the following properties:
\begin{itemize}
\item[\textbf{(P1)}] the Hamiltonian $K$ is non-degenerate;
\item[\textbf{(P2)}]the boundary values $a$ and $b$ of the action interval are not in the action spectrum $\mathcal{S}(K)$ of $K$;
\item[\textbf{(P3)}] we have $K\geq H$.
\end{itemize}
For the remaining part of this section we will always assume the above properties whenever we speak of perturbations $K$ of a Hamiltonian $H$.

The set of such perturbations is partially ordered by $K^1\leq K^0$ whenever $K_t^1(x)\leq K_t^0(x)$ for all $x\in M$ and $t\in S^1$. Consider a monotone decreasing homotopy $K^s$ from $K^0$ to $K^1$. 
By condition (P1), both perturbations $K^0$ and $K^1$ are non-degenerate. Thus we have an induced monotone homotopy map between the Floer homology groups, which are well-defined by (P1) and (P2).
In this case, this map is still a homomorphism, but it needs not be an isomorphism. Those monotone homotopy maps give rise to transition maps $\HF^{(a,\,b)}_{*}(K)\to\HF^{(a,\,b)}_{*}(\tilde{K})$ whenever $K\geq\tilde{K}$. Then we can define the filtered Floer homology of $H$ by
\[
\HF^{(a,\,b)}_{*}(H)=\lim\limits_{\longrightarrow}\HF^{(a,\,b)}_{*}(K)
\]
as the direct limit of homology groups.

\begin{rmk}
If $H$ is non-degenerate and $a$ and $b$ are not in the action spectrum $\mathcal{S}(H)$, this definition gives the ordinary filtered Floer homology of $H$, as $H$ can be viewed as the trivial perturbation of itself and thus as the smallest of all considered perturbations $K$.
\end{rmk}

By construction of filtered Floer homology for non-degenerate Hamiltonians, we have a long exact sequence of filtered Floer homology groups
\[
\cdots\rightarrow \HF_{*}^{(a,\,b)}(K)\rightarrow \HF_{*}^{(a,\,c)}(K)\rightarrow \HF_{*}^{(b,\,c)}(K)\rightarrow \HF_{*-1}^{(a,\,b)}(K)\rightarrow\cdots
\]
for any non-degenerate Hamiltonian $K$ with $a,b,c\notin\mathcal S(K)$. The maps of this exact sequence commute with the monotone homotopy map.
Then, for the limit function $H$, the analog sequence
\[
\cdots\rightarrow \HF_{*}^{(a,\,b)}(H)\rightarrow \HF_{*}^{(a,\,c)}(H)\rightarrow \HF_{*}^{(b,\,c)}(H)\rightarrow \HF_{*-1}^{(a,\,b)}(H)\rightarrow\cdots
\]
is also exact. This can be proved by a standard diagram chasing argument using the commutativity and the definition of the limit groups.

In the definition of the filtered Floer homology as a limit, we can also restrict the family of perturbations by requiring other properties in addition to (P1)-(P3). The restricted family of perturbations is sufficient to define the limit if they form a cofinal set, i.e. for any perturbation satisfying (P1)-(P3) we can find a smaller one with the additional properties. The limit then does not depend on the perturbations that do not have the additional properties.

In particular, we will later consider a cofinal set of perturbations for which the filtered Floer homology splits into a direct sum decomposition that is compatible with the monotone homotopy maps. Then we also have a direct sum decomposition of the limit.

We can also define monotone homotopy maps for homotopies starting from $H$.
Due to condition (P3), the monotone homotopy map for a homotopy starting from any perturbation $K$ factors through all perturbations closer to the limit function $H$ than $K$. Then we define the monotone homotopy map from $\HF^{(a,b)}(H)$ as the limit of monotone homotopy maps from the perturbations. The resulting map, still called monotone homotopy map, has the same properties as the usual homotopy maps.

\section{Direct sum decomposition in filtered Floer homology}
\label{sec:direct sum}

\subsection{The direct sum decomposition}
\label{sec:Bewsum}
In this section, we prove the existence of a direct sum decomposition of filtered Floer homology for short action intervals. This decomposition enables us to restrict our attention to the behavior of the Hamiltonian on a fixed open set. 

To construct this direct sum decomposition, let $K$ be a non-degenerate Hamiltonian on $M$. Consider two open sets $U$ and $V$ such that $U\subset V$ and both sets are bounded by level sets of $K$. 
On the shell $\bar{V}\setminus U$, assume that the Hamiltonian $K$ is autonomous and does not have one-periodic orbits. In particular, this implies that $U$ and $V$ are homotopy equivalent.
 Also fix an almost complex structure $J$ on $M$, which is compatible with $\omega$.

Consider the splitting of Floer chain groups into the direct sum
\begin{equation}
\CF_{*}^{(a,\,b)}(K)=\CF_{*}^{(a,\,b)}(K,U)\oplus \CF_{*}^{(a,\,b)}(K;M,U),
\label{chain sum}
\end{equation}
where the first summand is generated by the one-periodic orbits in $U$ with capping equivalent to a capping contained in $U$. The second summand is spanned by all the remaining capped orbits.

\begin{prop}
Let the Hamiltonian $K$ and the open sets $U$ and $V$ be as above.
There exists an $\epsilon>0$, depending only on $J$, the open sets $U$ and $V$ and on $K|_{V\setminus U}$ such that \eqref{chain sum} gives rise to a direct sum decomposition of homology
\begin{equation}
\HF_{*}^{(a,\,b)}(K)=\HF_{*}^{(a,\,b)}(K,U)\oplus \HF_{*}^{(a,\,b)}(K;M,U)
\label{Homology sum}
\end{equation}
whenever the action interval $(a,\,b)$ is chosen such that $b-a<\epsilon$.
\label{hom decomp}
\end{prop}

To prove this, we need to show that for such Hamiltonians $K$ no Floer trajectory can connect orbits from different summands, if the action interval is sufficiently small. The key to that is proving the following proposition which provides a lower bound on the energy for those Floer trajectories.

\begin{prop}
Let $K$ be a non-degenerate Hamiltonian and let $U$ and $V$ be open sets that are both bounded by level sets of $K$. Assume furthermore that $K$ does not have one-periodic orbits in $\bar{V}\setminus U$ and is autonomous on this shell.
Let $u\colon S^1\times\mathbb{R}\to M$ be a Floer trajectory that intersects $\partial U$ and $\partial V$.
Then there is a constant $\eps>0$, only depending on the open sets $U$ and $V$, the restriction of the Hamiltonian $K$ and the almost complex structure $J$ to $\bar{V}\setminus U$, such that $E(u)>\eps$.
\label{Energybound}
\end{prop}

\begin{proof}[Proof of Proposition \ref{hom decomp}]
Let $\bar{x}$ and $\bar{y}$ be two capped orbits in $\HF_{*}^{(a,\,b)}(K)$. Assume that $\bar{x}$ and $\bar{y}$ are connected by a Floer trajectory $u$, and let $\bar{x}$ be in $\HF_{*}^{(a,\,b)}(K,U)$. We need to show that $y$ is contained in $U$ and the capping of $\bar{y}$ is equivalent to a capping in $U$.

By construction, $V$ is homotopy equivalent to $U$ and the capping of $\bar{y}$ is equivalent to $u\#(\text{the capping of }\bar{x})$. Thus it suffices to show that the Floer trajectory $u$ is contained in $V$. If $u$ did leave $V$, it would have to intersect both boundary components of $V\setminus U$, as $u$ is converging to the orbit $x$, which is contained in $U$.
By Proposition~\ref{Energybound}, such a trajectory would have to have energy $E(u)>\eps$ for some constant $\eps>0$. Thus, if we choose the action interval $(a,\,b)$ such that $b-a$ is smaller than the lower bound $\eps$ in Proposition \ref{Energybound}, the Floer trajectory $u$ has to be contained in $V$ and we have the desired direct sum in homology.
\end{proof}

\begin{rmk}
In general, the direct sum decomposition from Proposition \ref{hom decomp} need not be compatible with monotone homotopy maps. In some important cases, however, this is the case.
For example, consider two Hamiltonians $K^1$ and $K^2$ that agree on $V\setminus U$ up to a constant and assume $K^1\geq K^2$. Then the above direct sum decomposition is compatible with the monotone homotopy map $\HF_{*}^{(a,\,b)}(K^1)\to \HF_{*}^{(a,\,b)}(K^2)$.
Indeed, the monotone homotopy map is defined using a version of the Floer equation. If the two Hamiltonians agree up to a constant, their Hamiltonian vector fields agree and this equation is exactly the standard Floer equation. 
Thus the above proof of Proposition \ref{hom decomp} also applies in this setting and shows that the monotone homotopy map is compatible with the direct sum decomposition for sufficiently small action intervals.
\label{compatible}
\end{rmk}

\begin{cor}
Let $K$ be any Hamiltonian, not necessarily non-degenerate.
Assume that the open sets $U$ and $V$ are bounded by level sets of $K$. If the Hamiltonian $K$ is autonomous on $V\setminus U$ and does not have periodic orbits in this shell, then for sufficiently small action interval $(a,\,b)$ the direct sum decomposition \eqref{Homology sum} holds.
\end{cor}
\begin{proof}
It suffices to construct a cofinal set of non-degenerate perturbations of $K$, such that the direct sum decomposition (3.2) holds for all of those Hamiltonians and is compatible with the monotone homotopy maps.

Consider the perturbations that differ from $K$ on $V\setminus U$ only by a constant. These form a cofinal set, since for every perturbation $H\geq K$ we can find a smaller one with that additional property. We can choose these perturbations to be non-degenerate, as $K$ does not have periodic orbits in $\bar{V}\setminus U$ and  there are no restrictions on the perturbation outside $\bar{V}\setminus U$.
The connecting maps between the Floer homologies of the perturbations are monotone homotopy maps and respect the direct sum decomposition. Thus we also have a direct sum in the limit.
\end{proof}

\subsection{Energy estimates and the proof of Proposition \ref{Energybound}}
\label{sec:energy}

To prove the proposition, we need to find a lower bound for the energy of Floer trajectories crossing the shell $V\setminus \bar{U}$. The first lemma can be used to bound the time-integral in the expression for the energy away from zero for the part of a Floer trajectory in a compact set not containing one-periodic orbits.

\begin{lemma}
\label{lowbound}
Let $W$ be a bounded open set with smooth boundary and at least two boundary components and let $K$ be an autonomous Hamiltonian on $\bar{W}$. Assume that $K$ is constant on each boundary component and does not have one-periodic orbits in $\bar{W}$.
Then there exists a constant $C_1>0$, depending only on the almost complex structure $J$, the open set $W$ and $K$ such that:
\begin{enumerate}
\item For $T\leq 1$, any path $\gamma\colon[0,\,T]\to \bar{W}$, which connects two distinct boundary components of $W$, satisfies
\[
\int_0^T\left\|\dot{\gamma}(t)-X_K(\gamma(t))\right\|^2\ dt >C_1.
\]
\item Any loop $\gamma\colon S^1\to \bar{W}$ satisfies
\[
\int_{S^1}\left\|\dot{\gamma}(t)-X_K(\gamma(t))\right\|^2\ dt >C_1.
\]
\end{enumerate}
\end{lemma}

This lemma is a generalization of lemmas in \cite{Us}, but the existence of similar lower bounds goes back to \cite{Sa}. The proof given in Section \ref{sec:Bewlowbound}, however, differs from the proofs in \cite{Sa,Us}. With $W=V\setminus\bar{U}$, this lemma implies the proposition if the area of $u^{-1}(W)$ is small. If this area is not small, we need the following lemma to relate the area of the domain and the energy for certain parts of a Floer trajectory.

\begin{lemma}[Usher's lemma]
\label{Usher2.3}
Let $W$ be a bounded open set with smooth boundary and at least two boundary components and let $K$ be an autonomous Hamiltonian on $\bar{W}$.
 Let $S$ be a connected subset of the cylinder $S^1\times\mathbb{R}$ and let $u\colon S\to \bar{W}$ satisfy the Floer equation \eqref{Floer} with Hamiltonian $K$. 
Assume that $u(\partial S)\subseteq \partial W$. If $u(S)$ intersects two distinct boundary components of $W$, then there exists a constant $C_2$, depending only on the domain $W$, the Hamiltonian $K$ and the complex structure $J$ on $W$, such that
\[
Area(S)+E(u)\geq C_2.
\]
\end{lemma}

This lemma is a generalization of Lemma 2.3 in \cite{Us}. We prove this lemma in Section~\ref{sec:bewUsher} and continue here with the proof of Proposition~\ref{Energybound}. Similarly to the special case in \cite{Us}, both lemmas are used with $W$ being the shell between two open sets to bound the energy of certain Floer trajectories away from zero. 

Pick two open sets $U'$ and $V'$ bounded by level sets of $K$ such that 
\[
U\subset U'\subset V'\subset V.
\]
Denote the loop $t\mapsto u(t,\,s)$ for fixed $s$ by  $\gamma_s(t)$ and consider the set
\[Z=\left\{s\in\mathbb{R}\ |\ \gamma_s \text{ intersects } V'\setminus U'\right\}.
\]

Then for  every $s\in Z$, we either have $\gamma_s\subseteq V\setminus U$ or $\gamma_s$ intersects one of the boundary components of $V\setminus U$. In the first case we can apply 
Lemma \ref{lowbound} $(ii)$ to the Hamiltonian $K$ and $W=V\setminus \bar{U}$.

In the second case, the path $\gamma_s$ also intersects one of the boundary components of $V'\setminus U'$ and we can apply Lemma \ref{lowbound} $(i)$ with $W$ taken to be one of the shells $V\setminus V'$ or $U'\setminus U$.
Denote by $C$ the minimum of the constants $C_1$ from Lemma~\ref{lowbound} for the shells $V\setminus U$, $V\setminus V'$ and $U'\setminus U$.

Then we have the following estimate for the energy of $u$:
\begin{eqnarray*}
E(u)&=&\int_{\mathbb{R}}\int_{S^1}\left\|\partial_s u\right\|^2\ dt\ ds\\
&\geq&\int_{Z}\int_{S^1}\left\|\partial_t u(s,\,t)-X_K(u(s,\,t))\right\|^2\ dt\ ds\\
&\geq&\int_Z C\ ds=C\ m_{Leb}(Z).
\end{eqnarray*}

If $m_{Leb}(Z)\geq C_2/2$, where $C_2$ is the constant from Lemma \ref{Usher2.3} for the shell $W=V'\setminus U'$, we have a lower bound $C C_2/2$ for the energy of $u$.

If $m_{Leb}(Z)<C_2/2$, we choose $S$ as one connected component of $u^{-1}(V'\setminus U')$, such that $u(S)$ intersects both boundary components. Since $u$ intersects $\partial U$ and $\partial V$, such a set $S$ exists and $S\subseteq Z\times S^1$. Then we have $Area(S)\leq m_{Leb}(Z)\leq C_2/2$ and $u(\partial S)\subseteq \partial (V'\setminus U')$.
Now Lemma \ref{Usher2.3} applies with $W=V'\setminus \bar{U'}$ and we find that
\[
E(u)\geq E(u|_{S})\geq C_2-Area(S)\geq C_2/2.
\]
Thus with $\eps=\min\{CC_2/2,\,C_2/2\}$ we have found a lower bound for the energy in either case.

\subsection{Proof of Lemma \ref{lowbound}}
\label{sec:Bewlowbound}

By the Schwarz inequality, we have
\begin{eqnarray*}
\int_0^T\left\|\dot{\gamma}(t)-X_K(\gamma(t))\right\|^2\ dt &\geq&\left(\int_0^T\left\|\dot{\gamma}(t)-X_K(\gamma(t))\right\|\ dt\right)^2
\end{eqnarray*}
 for $T\leq 1$ and it suffices to find a lower bound for the $L^1$-norm.

To that end, for a path $\gamma(t)$ in $\bar{W}$, we define the path $\eta(t)=\varphi_K^{-t}(\gamma(t))$. 
By the chain rule we have 
\begin{eqnarray*}
\dot\gamma(t)&=&d\varphi_K^t(\eta(t))\dot\eta(t)+\big(\frac{d}{dt}\varphi_K^t\big)(\eta(t))\\
&=&d\varphi_K^t(\eta(t))\dot{\eta}(t)+X_K(\gamma(t)). 
\end{eqnarray*}

Recall for part $(i)$ that we assume $K$ to be autonomous and constant on the boundary components of $W$. The two boundary components of $W$ are preserved under the flow. Since $\gamma$ connects two distinct boundary components of $W$, the same is true for $\eta$. Denote the distance of these boundary components by $\delta$.
Then we find the lower bound by the following calculation:
\begin{eqnarray*}
\int_0^T\left\|\dot{\gamma}(t)-X_K(\gamma(t))\right\|\ dt&=&\int_0^T \|d\varphi_K^t(\eta(t))\dot\eta(t)\| \ dt\\
&>&c\cdot \int_0^T \|\dot\eta(t)\| \ dt\\
&\geq&c\cdot d(\eta(0),\,\eta(T))\\
&\geq&c\cdot \delta.
\end{eqnarray*}
The constant $c$ is positive since $K$ is smooth and both $t$ and $\eta(t)$ are varying in compact sets and $K$ has no critical points in $\bar{W}$. 

Similarly, we find for part $(ii)$ 
\begin{eqnarray*}
\int_{S^1}\left\|\dot{\gamma}(t)-X_K(\gamma(t))\right\|\ dt&=&\int_0^1 \|d\varphi_K^t(\eta(t))\dot\eta(t)\| \ dt\\
&\geq&c\cdot d(\eta(0),\eta(1))\\
&=&c\cdot d\big(\gamma(0),\varphi_K^{-1}(\gamma(0))\big).
\end{eqnarray*}
As $\bar{W}$ is compact and $\varphi_K$ is continuous with no one-periodic orbits in $\bar{W}$, this distance is bounded away from zero.

Thus in both parts we have found a lower bound and we set $C_1$ to be the minimum of those bounds.

\subsection{Proof of Usher's lemma}
\label{sec:bewUsher}
For simplicity of notation, we assume that $W$ has exactly two boundary components.
Recall that we want to prove the existence of a lower bound for $Area(S)+E(u)$ which is independent of $u$. Since $u$ is a solution of the Floer equation, the graph $\tilde{u}\colon S\to S^1\times\mathbb{R}\times\bar{W}$ is a $\tilde{J}$-holomorphic curve for a certain almost complex structure $\tilde{J}$ which is tamed by $\tilde{\omega}=ds\wedge dt-dt\wedge dK+\omega$. (For the precise definition of $\tilde{J}$, see the proof of Lemma 2.3 in \cite{Us}). 
This almost complex structure $\tilde{J}$ depends only on the almost complex structure $J$ on $\bar{W}$ and on the Hamiltonian $K$ on $\bar{W}$.

For any subset $S'\subseteq S$, this definition of $\tilde{\omega}$ gives
\[
\int_{S'}\tilde{u}^{*}\tilde{\omega}=\int_{S'}\ ds\wedge dt + \int_{S'}\left|\frac{\partial u}{\partial s}\right|_{J_t}^2\ ds\ dt=Area(S')+E(u|_{S'}).
\]

Let $\Sigma$ be a closed hypersurface in $W$ which separates the two boundary components of $W$. By assumption, $u$ intersects both boundary components. Thus there exists a $z_0\in S$ such that $u_0=u(z_0)\in\Sigma$. 
We now choose a ball $B\subseteq W$ centered at $u_0$ and a disk $D\subseteq S^1\times \R$ with fixed radius centered at $z_0$.

Now we consider a ball $\tilde{B}$ centered at $(z_0,u_0)$ and contained in $D\times B$. Since the radius of $D$ is fixed, the radius of this ball $\tilde{B}$ depends only on the radius of $B$ and thus only on the open set $W$.
Then we define
\[
\tilde{S}=\{z\in S\ |\ \tilde{u}(z)\in\tilde{B}\}.
\]

By definition, the boundary of $\tilde{S}$ is mapped to the boundary of $\tilde{B}$. Indeed,  $\tilde{u}(\partial S)$ is contained in $(\partial S)\times(\partial W)$ and therefore not in $\tilde{B}\subseteq D\times B$.

Let us now view the graph of $u$ as a map $\tilde{u}\colon\tilde{S}\to\tilde{B}$. As $B$ is contained in $W$, where $K$ is fixed, the complex structure $\tilde{J}$ on $D\times B$ depends only on the ball $B$, the Hamiltonian $K|_{W}$ and the complex structure $J$ on $W$.

By definition of $\tilde{S}$, the center $(z_0,\ u_0)=\tilde{u}(z_0)$ of $\tilde{B}$ is contained in the image of $\tilde{u}$. Now $\tilde{u}$ is considered to be a $\tilde{J}$-holomorphic curve in $\tilde{B}$, which is passing through the center and has no boundary in the interior of $\tilde{B}$.
For this $\tilde{u}$, Proposition 4.3.1(ii) in \cite{Si} applies and we have $Area (\tilde{S})+E(u|_{\tilde{S}})\geq C_2(u_0)$. 

This constant $C_2(u_0)$ still depends on $u$, since the choice of the center $u_0$ of the ball $B$ depends on $u$. To obtain a constant that is independent of $u$, we take the infimum over all possible $u_0$ and define
\[
C_2=\inf \{C_2(u_0)\ |\ u_0\in\Sigma\}.
\]
Since the hypersurface $\Sigma$ is compact, this constant $C_2$ is positive and independent of $u$ with $Area(S)+E(u)\geq Area(\tilde{S})+E(u|_{\tilde{S}})\geq C_2>0$.

\section{Proof of  Theorem \ref{theorem}}
\label{sec:Beweis}
\subsection{Outline of the proof}
\label{sec:outline}

The key to proving Theorem \ref{theorem} is a geometrical description of \SDMs\ given in Proposition \ref{SDM geo}. In particular, we can assume that the \SDM\ is a constant orbit $p$ with trivial capping. Furthermore,  we can assume that $p$ is a strict local maximum of $H$ and that $H$ has arbitrarily small Hessian at $p$.

Then we use the squeezing method from \cite{BPS,Gi,GG1} and construct Hamiltonians $H_+$ and $H_-$ such that $H_-<H<H_+$.
It suffices to show that a linear homotopy from $H_+$ to $H_-$ induces a non-zero map between the filtered Floer homology groups of $H_\pm$ for the action interval in question, since this map factors through the filtered Floer homology of $H$. Then the Floer homology group of $H$ cannot be trivial.

As functions of the distance from $p$, the functions $H_+$ and $H_-$ are constructed similarly to the ones used in \cite{Gi,GG-gaps}; see Section \ref{sec:functions} for details.

For the Hamiltonians $H_\pm$ we use the direct sum decomposition from Proposition~\ref{hom decomp}. By construction of $H_\pm$ and Remark \ref{compatible}, this will be compatible with the limit construction of filtered Floer homology and the monotone homotopy map for a homotopy from $H_+$ to $H_-$.
To prove that the monotone homotopy map is non-zero, it suffices to show that the restriction to one of the summands is an isomorphism.

We will consider the summand $\HF_*(H_\pm,\,U)$ for a neighborhood $U$ of the \SDM\ $p$, as this depends only on the functions restricted to $U$ and the symplectic structure in $U$ and is independent of the ambient manifold. Thus we can view $U$ as an open set in any symplectic manifold and the theorem follows as in the symplectically asherical case in \cite{Gi}.

\begin{rmk}
This process of localizing the problem is fundamentally different from the localization in the definition of local Floer homology. 
Here we only fix the Hamiltonian on a shell $V\setminus U$ between two bounded open sets $U$ and $V$. Then we use the small action interval, and thus small energy of Floer trajectories, to ensure that the trajectories do not leave $V$ using the energy bounds from Section~\ref{sec:energy}.

For local Floer homology we do not directly restrict the action interval but fix the Hamiltonian outside an open set $U$ that only contains one isolated one-periodic orbit $\bar x$. Then we take a small non-degenerate perturbation of the Hamiltonian to split this one-periodic orbit up into non-degenerate periodic orbits. The actions of those are close to the action of $\bar x$ and thus the energy of Floer trajectories connecting them is small.
As the Hamiltonian is fixed outside $U$, this ensures that Floer trajectories between orbits in $U$ stay in $U$. Then the local Floer homology is defined by restricting the definition of Floer homology to $U$.
\end{rmk}

\subsection{Geometric characterization of \SDMs}
\label{sec:sdm}

In this section we state some geometric properties of \SDMs. The existence of a \SDM\ enters the proof of Theorem \ref{theorem} only via those properties.

For the formulation of the geometric characterization of a \SDM\ we first need to recall the definition of the norm of a tensor with respect to a coordinate system.
By definition, on a finite-dimensional vector space the norm $\left\|v\right\|_{\Xi}$ of a tensor $v$ with respect to a coordinate system $\Xi$ is the norm of $v$ with respect to the inner product for which $\Xi$ is an orthonormal basis. For a coordinate system $\xi$ on a manifold $M$ near a point $p$, the natural coordinate basis in $T_pM$ is denoted by $\xi_p$.

\begin{prop}[\cite{GG-gaps,GG-loc}]
\label{SDM geo}
Let $\bar{x}$ be a \SDM\ of a Hamiltonian $H$ and let $p=x(0)\in M$. Then there exists a sequence of contractible loops $\eta_i$ of Hamiltonian diffeomorphisms such that $x(t)=\eta_i^t(p)$, i.e each loop $\eta_i$ sends $p$ to $x$.
Furthermore, the Hamiltonians $K^i$ given by $\varphi_H^t=\eta_i^t\circ\varphi_{K^i}^t$ and the loops $\eta_i$ satisfy the following conditions:
\begin{itemize}
\item [\textbf{(K1)}] The point $p$ is a strict local maximum of $K_t^i$ for $t\in S^1$.
\item [\textbf{(K2)}] There exist symplectic bases $\Xi^i$ of $T_pM$ such that 
\[
\left\|d^2(K_t^i)_p\right\|_{\Xi^i}\rightarrow 0 \text{ uniformly in } t\in S^1.
\]
\item [\textbf{(K3)}] The loop $\eta_i^{-1}\circ \eta_j$ has identity linearization at $p$ for all $i$ and $j$ (i.e. for all $t\in S^1$ we have $d\big((\eta_i^t)^{-1}\circ \eta_j^t\big)_p=I$), and is contractible to $id$ in the class of such loops.
\end{itemize}
\end{prop}

A proof of this proposition and also of the fact that this description is equivalent to the definition of \SDMs\ can be found in \cite{GG-gaps,GG-loc}. It is also shown there that the conditions (K1) and (K2) already imply (K3) as a formal consequence.

When the concept of \SDMs\ was introduced in~\cite{Hi} by Hingston and in the first formal definition given in \cite{Gi}, this characterization was used as a definition of \SDMs.

\begin{rmk}
The loops $\eta_i^{-1}\circ \eta_j$ are loops of Hamiltonian diffeomorphisms fixing $p$.
The construction in \cite{Gi} shows that the loops $\eta_i$ can be chosen such that $\eta_i^{-1}\circ \eta_j$ are supported in an arbitrarily small neighborhood of $p$.
\label{supp eta}
\end{rmk}

\subsection{The functions $H_+$ and $H_-$}
\label{sec:functions}

By Proposition \ref{SDM geo} above it suffices to prove the theorem for the function $K^1$ and the constant orbit $p$ with trivial capping as \SDM. We keep the notation $H$ for $K^1$. Fix a Darboux chart in a neighborhood $W$ of $p$ such that $p$ is a strict global maximum of $H$ on $W$. 
We also fix now an almost complex structure $J$ on $M$ that is compatible with $\omega$.

Let $U$ and $V$ be balls centered at $p$ and contained in $W$. 
We then construct the function $H_{+}$ and an auxiliary function $F$ to be of the form shown in Figure \ref{fig:functions}.

More concretely, we fix balls 
\[
B_{r_-}\subset B_{r_+}\subset B_r\subset U\subset V\subset B_R\subset B_{R_-}\subset B_{R_+}\Subset W.
\]

Then the function $H_+$ is defined to be radially symmetric around $p$ with the following properties:
\begin{itemize}
\item $H_+\geq H$ and $H_+\equiv c=H(p)$ on $B_{r_-}$.
\item On $B_{r_+}\setminus B_{r_-}$ the function $H_+$ is monotone decreasing.
\item On $B_r\setminus B_{r_+}$ the function is constant.
\item In the shell $B_R\setminus B_r$ the function is monotone increasing, linear as a function of the square of the distance from $p$ with small slope $\alpha$ on $V\setminus U$ such that there are no one-periodic orbits in $V\setminus B_r$.
\item The function $H_+$ is again constant on $B_{R_-}\setminus B_R$ with a value less than $c$.
\item It is monotone increasing on $B_{R_+}\setminus B_{R_-}$
\item Outside $B_{R_+}$ the function $H_+$ is constant and equal to its maximum.
\end{itemize}

\begin{figure}
\centering
\scalebox{0.4}[0.2]{
   \includegraphics{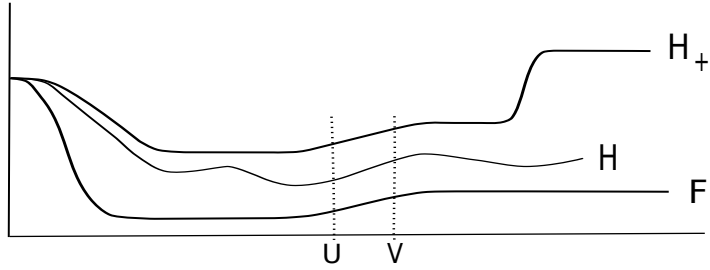}}
\caption{The functions $H_+$ and $F$ as functions of the distance from $p$}
\label{fig:functions}
\end{figure}   

Inside $B_r$ and outside $B_R$ the construction of $H_+$ is exactly as in \cite{Gi}. The linear part in $V\setminus U$ ensures that Proposition \ref{hom decomp} applies. 

More concretely, we first choose some small constant $\alpha_0>0$ such that $\alpha_0/\pi$ is irrational. Then we fix the Hamiltonian $H_+$ on $B_r$ and pick $\eps>0$ smaller than the energy bound from Proposition~\ref{hom decomp} for a Hamiltonian linear with slope $\alpha_0$ on $V\setminus U$.
Using these choices we take a sufficiently large order of iteration $k$ as in \cite{Gi,GG-gaps}. 
Furthermore, we now fix $H_+$ outside $B_r$ with slope $\alpha=\alpha_0/k$ on $V\setminus U$.
We thus have the direct sum decomposition of filtered Floer homology by Proposition~\ref{hom decomp} for $H_+^{(k)}$. At this point we choose some $\delta_k\in(0,\,\eps/2)$, depending on $k$, to ensure that the action intervals intervals $(kc+\delta_k,\,kc+\eps)$ and $(kc-\delta_k,\,kc+\delta_k)$ are sufficiently small for the direct sum decomposition.

Let us now turn to the construction of $H_-$ using the existence of a \SDM. The geometrical characterization of \SDMs\ in Proposition \ref{SDM geo} and Remark \ref{supp eta} imply that we have
\begin{itemize}
\item a loop $\eta^t$ of Hamiltonian diffeomorphisms fixing $p$, which is supported in~$U$,
\item and a system of coordinates $\xi$ on a neighborhood $W$ of $p$
\end{itemize}
such that the Hamiltonian $K$ generating the flow $\eta^{-t}\circ\varphi_H^t$ has a strict local maximum at $p$ and $\max_t\left\|d^2(K_t)_p\right\|_{\xi_p}$ is sufficiently small. The loop $\eta$ is contractible in the class of loops having identity linearization at $p$.
Let $G_s^t$ be a Hamiltonian generating such a homotopy $\eta_s^t$, normalized by $G_s^t(p)\equiv 0$. We then normalize $K$ by the additional requirement that $K_t(p)\equiv c$ (or equivalently that $H=G\# K$).

Then there exists a function $F$, depending on the coordinate system $\xi$, such that
\begin{itemize}
\item $\left\|d^2F_p\right\|_{\xi_p}$ is sufficiently small,
\item $F\leq K$ and $F(p)=c=H(p)$ is the global maximum of $F$.
\end{itemize}
To be more precise, in $B_r$ the function $F$ is a bump function centered at $p$, constant outside $B_R$ and differs from $H_+$ only by a constant on $V\setminus U$. The last condition is necessary to have the direct sum decompositions of the filtered Floer homology groups of $H_+$ and $F$ be compatible with the monotone homotopy map for a linear homotopy from $H_+$ to $F$.

Then $F^s=G^s\# F\leq H_+$ is an isospectral homotopy with $F^1=F$, i.e., a homotopy such that the action spectrum $\mathcal{S}(F^s)$ is independent of $s$.
We define the function $H_-$ by
\[
H_-:=G^0\# F\leq G^0\# K=H.
\]
Since $\eta$ is supported in $U$, the function $G^0$ is constant outside $U$. Then $H_-$ differs from $F$ and $H_+$ only by the constant value of $G$ on $\bar V\setminus U$. Therefore we also have the direct sum decomposition from Proposition \ref{hom decomp} for $H_-$. It is compatible with the homomorphism induced by the homotopy $F^s$ and the monotone homotopy map for a homotopy from $H_+$ to $H_-$.

\subsection{The Floer homology of $H_{\pm}$ and the monotone homotopy map}
\label{sec:HF}
In the symplectically aspherical case, the filtered Floer homology groups for $F$ and $H_+$ in the action intervals in question have been calculated in \cite{Gi,GG1}. 
Using this, we obtain an isomorphism
\[
\mathbbm{Z}_2\cong \text{HF}_{n+1}^{(kc+\delta_k,\,kc+\epsilon)}(H_+^{(k)},U)\to \text{HF}_{n+1}^{(kc+\delta_k,\,kc+\epsilon)}(F^{(k)},U)\cong\mathbbm{Z}_2
\]
by the same argument as in \cite{Gi}, since this summand behaves exactly like the filtered Floer homology in the symplectically aspherical case.
Then we have the commutative diagram
\begin{displaymath}
\xymatrix{
\ar[d]_{\Psi}\HF_{n+1}^{(kc+\epsilon,\,kc+\delta_k)}(H_+^{(k)},U)\ar[rd]^{\cong}&\\
\HF_{n+1}^{(kc+\epsilon,\,kc+\delta_k)}(H_-^{(k)},U)\ar[r]^{\cong}&\HF_{n+1}^{(kc+\epsilon,\,kc+\delta_k)}(F^{(k)},U)
}
\end{displaymath}
where  the horizontal map is induced by the isopsectral homotopy $F^s$ and the other maps are monotone homotopy maps. As in the symplectically aspherical case, the isospectral homotopy induces an isomorphism in this summand of the filtered Floer homology.
The commutativity is established the same way as in the symplectically aspherical case. 
(Observe that it is essential to use the "localized" Floer homology for all Hamiltonians in the diagram. For the full filtered Floer homology groups, the standard argument for the commutativity of the analogous diagram does not apply in the case of a symplectically irrational manifold.)
The diagonal map is an isomorphism by the same argument as in \cite{Gi} using the long exact sequence of filtered Floer homology to go over to the action interval $(kc-\delta_k,kc+\delta_k)$.
By the commutativity of this diagram, the map $\Psi$ is also an isomorphism. Thus the monotone homotopy map
\[
\HF_{n+1}^{(kc+\delta_k,\,kc+\epsilon)}(H_+^{(k)})\to \HF_{n+1}^{(kc+\delta_k,\,kc+\epsilon)}(H_-^{(k)})
\]
 is non-zero and this map factors through the Floer homology group of $H$, which we want to show to be non-trivial. This proves Theorem~\ref{theorem}.

\end{document}